\documentclass[letterpaper,dvipsnames,12pt]{article}

\usepackage[margin=1in]{geometry}
\usepackage[utf8]{inputenc}
\usepackage[T1]{fontenc}
\usepackage{amsmath,amsthm,amssymb}
\usepackage[
    colorlinks=true,
    linkcolor=Maroon,
    citecolor=MidnightBlue,
    urlcolor=MidnightBlue
]{hyperref}
\usepackage{xcolor}
\usepackage{enumerate}
\usepackage{amsrefs}
\usepackage{colonequals}
\usepackage{microtype}

\let\Re\relax
\DeclareMathOperator{\Re}{Re}
\let\Im\relax
\DeclareMathOperator{\Im}{Im}
\DeclareMathOperator{\Law}{Law}
\DeclareMathOperator{\diag}{diag}
\DeclareMathOperator{\Tr}{Tr}
\DeclareMathOperator{\Cov}{Cov}
\DeclareMathOperator{\Var}{Var}
\DeclareMathOperator{\Pois}{Pois}
\DeclareMathOperator{\supp}{supp}

\let\P\relax
\newcommand{\P}{\mathbb{P}}
\newcommand{\C}{\mathbb{C}}
\newcommand{\R}{\mathbb{R}}
\newcommand{\N}{\mathbb{N}}
\newcommand{\Z}{\mathbb{Z}}
\newcommand{\E}{\mathbb{E}}

\newcommand{\bX}{\mathbf{X}}
\newcommand{\bY}{\mathbf{Y}}
\let\d\relax
\newcommand{\d}{\mathrm{d}}
\newcommand{\dis}{\rho}
\newcommand{\eps}{\varepsilon}
\let\complement\relax
\newcommand{\complement}{\mathrm{c}}

\newcommand{\cC}{\mathcal{C}}
\newcommand{\cH}{\mathcal{H}}
\newcommand{\cJ}{\mathcal{J}}
\newcommand{\cN}{\mathcal{N}}
\newcommand{\cR}{\mathcal{R}}
\newcommand{\cW}{\mathcal{W}}
\newcommand{\cT}{\mathcal{T}}

\newcommand{\hH}{\widehat{H}}
\newcommand{\hX}{\widehat{X}}
\newcommand{\hq}{\widehat{q}}

\renewcommand{\le}{\leqslant}
\renewcommand{\ge}{\geqslant}

\newcommand{\tv}{\mathrm{d}_\mathrm{TV}}
\newcommand{\ind}{\mathbf{1}}
\newcommand{\LS}{\Lambda_{\mathrm{slow}}}
\newcommand{\tmix}{t_{\mathrm{mix}}^{(n)}}
\newcommand{\abs}[1]{\left|#1\right|}

\newcommand{\fnorm}[1]{\left\lVert#1\right\rVert_{\mathrm{F}}}
\newcommand{\snorm}[1]{\left\lVert#1\right\rVert_{\star}}
\newcommand{\rnorm}[1]{\left\lVert#1\right\rVert_{\mathrm{row}}}
\newcommand{\onenorm}[1]{\left\lVert#1\right\rVert_1}
\newcommand{\twonorm}[1]{\left\lVert#1\right\rVert_2}
\newcommand{\opnorm}[1]{\left\lVert#1\right\rVert_{\mathrm{op}}}
\newcommand{\inner}[2]{\left\langle#1, #2\right\rangle}
\newcommand{\finner}[2]{\left\langle#1, #2\right\rangle_{\mathrm{F}}}
\newcommand{\set}[2]{\left\{#1 \colon #2\right\}}
\newcommand{\brac}[1]{\left(#1\right)}
\newcommand{\sq}[1]{\left[#1\right]}

\newtheorem{ccounter}{ccounter}[section]
\newtheorem{thm}[ccounter]{Theorem}
\newtheorem{lem}[ccounter]{Lemma}

\newtheorem{prop}[ccounter]{Proposition}

\title{Cutoff profile for the balanced generalized Bernoulli--Laplace model}
\author{Ritesh Goenka\thanks{Mathematical Institute, University of Oxford, Oxford, United Kingdom, Email: goenka@maths.ox.ac.uk} \and Kyprianos-Iason Prodromidis\thanks{Department of Mathematics, Princeton University, Princeton, USA, Email: kp2702@princeton.edu}}
\date{}

\begin{document}

\maketitle

\begin{abstract}
    \noindent We study the balanced generalized Bernoulli--Laplace model introduced by Goenka, Hermon, and Schmid (2025) and establish a Gaussian cutoff profile under spectral conditions satisfied by all reversible chains and the cyclic shift chain. Our method adapts the approach of Olesker-Taylor and Schmid (2026) for the two-urn model, combining concentration around a deterministic trajectory before the cutoff window with an Ornstein--Uhlenbeck approximation near equilibrium. We also give an instance of this model that exhibits cutoff but has no limit profile. Further, we determine the mixing time up to an additive constant for every irreducible instance of the model, identifying a $\log \log n$ correction governed by the largest Jordan block size corresponding to the slowest decaying modes.
\end{abstract}

\section{Introduction}

The \emph{cutoff phenomenon} describes an abrupt transition to near-equilibrium for a sequence of Markov chains, occurring over a time window that is asymptotically negligible when compared with the mixing time. In recent years, increasing attention has been paid to the finer question of determining the \emph{limit profile}, which describes the limiting shape of the total variation distance within the cutoff window. Olesker-Taylor and Schmid~\cite{OTS} recently developed an approach to determining the limit profile using diffusion approximation, applying it to the Bernoulli--Laplace urn model. Their method combines concentration around a deterministic trajectory up to shortly before cutoff with convergence to an Ornstein--Uhlenbeck diffusion near equilibrium. Karageorgiou and Prodromidis~\cite{KP} subsequently adapted this approach to a two-dimensional process to find the limit profile in the high-temperature Curie--Weiss model.

The classical Bernoulli--Laplace model consists of two urns containing $n$ balls each, with $n$ red and $n$ black balls in total. At each step, a ball is chosen uniformly from each urn and the two selected balls are exchanged. Diaconis and Shahshahani~\cite{DS} established cutoff at $\frac{1}{4}n\log n$ steps, with a window of order $n$. More recently, Goenka, Hermon, and Schmid~\cite{GHS} introduced a multi-urn model encompassing several earlier generalizations, in which one ball is chosen uniformly from each urn and the selected balls are redistributed according to a random permutation. We study the balanced case of their model, with an equal number of urns and colors. In this setting, for a fixed number of urns and a fixed distribution on permutations, they established cutoff whenever the resulting Markov chain is irreducible.

Formally, fix an integer $d \ge 2$ and a probability measure $\mu$ on $S_d$. For each integer $n \ge 2$, consider $d$ urns containing $n$ balls each, with $d$ colors and $n$ balls of each color in total. At rate $n$, one ball is chosen uniformly from each urn, and a permutation $\sigma \sim \mu$ is sampled independently of these choices. For each $j \in [d]$, the selected ball from urn $j$ is then moved to urn $\sigma(j)$. This normalization ensures that each ball is selected at rate one. We write $\bX^{(n)} = (X^{(n)}(t))_{t \ge 0}$ for the resulting Markov chain, where $X_{i,j}^{(n)}(t)$ denotes the number of balls of color $i$ in urn $j$ at time $t$. Throughout, we assume that the measure $\mu$ is \emph{irreducible}, meaning that for every $i, j \in [d]$, there exist $m \in \N$ and $\sigma_1, \dots, \sigma_m \in \supp(\mu)$ such that $\sigma_1 \cdots \sigma_m(i) = j$. By Lemma~3.2 of \cite{GHS}, this is equivalent to the irreducibility of $\bX^{(n)}$. Throughout, all constants, including those implicit in asymptotic notation, may depend on $d$ and $\mu$. 

The state space $\Omega_n$ of $\bX^{(n)}$ consists of all $d \times d$ matrices with nonnegative integer entries and all row and column sums equal to $n$. Let $\pi_n$ denote the unique stationary distribution of $\bX^{(n)}$, and write $H_t^{(n)}$ for its transition kernel, so that for $t \ge 0$ and $x, y \in \Omega_n$, we have
\begin{equation*}
    H_t^{(n)}(x, y) = \P\brac{X^{(n)}(t) = y \,\mid\, X^{(n)}(0) = x}.
\end{equation*}
The worst-case total variation distance to stationarity at time $t \ge 0$ is
\begin{equation*}
    d_n(t) \colonequals \max_{x \in \Omega_n} \tv\brac{H_t^{(n)}(x, \cdot), \pi_n},
\end{equation*}
where $\tv(\nu, \pi)$ denotes total variation distance. For $\eps \in (0, 1)$, define the $\eps$-mixing time by
\begin{equation*}
    \tmix(\eps) \colonequals \inf \set{t \ge 0}{d_n(t) \le \eps}.
\end{equation*}
We say that the sequence of chains $\{\bX^{(n)}\}_{n \ge 2}$ exhibits cutoff if, for every $\eps \in (0, 1/2)$,
\begin{equation*}
    \lim_{n \to \infty} \frac{\tmix(\eps)}{\tmix(1 - \eps)} = 1.
\end{equation*}

The motion of a fixed ball across the urns defines a Markov chain on $[d]$, which we call the \emph{single-ball chain}. Its transition
matrix at selection times is $U \in \R^{d \times d}$, given by
\begin{equation*}
    U_{i,j} \colonequals
    \sum_{\sigma \in S_d} \mu(\sigma)
    \ind_{\{\sigma(i) = j\}}
\end{equation*}
for $i, j \in [d]$. The matrix $U$ is doubly stochastic and, by our assumption on $\mu$, irreducible.

Our main theorem identifies a Gaussian cutoff profile under two spectral conditions on $U$. Let $\Lambda$ denote the spectrum of $U$, with eigenvalues counted with algebraic multiplicity, and set
\begin{equation*}
    \gamma \colonequals 1 - \max_{\lambda \in \Lambda \setminus \{1\}} \Re(\lambda).
\end{equation*}
The Perron--Frobenius theorem implies $\gamma > 0$. We call an eigenvalue $\lambda \ne 1$ \emph{slow} if $\Re(\lambda) = 1 - \gamma$ and \emph{fast} otherwise, and write $\LS$ for the multiset of slow eigenvalues. When every slow eigenvalue has equal algebraic and geometric multiplicities, we may choose an invertible matrix $S \in \C^{d \times d}$ such that
\begin{equation*}
    U = S \diag\brac{1, J_{\mathrm{slow}}, J_{\mathrm{fast}}} S^{-1},
\end{equation*}
where $J_{\mathrm{slow}}$ is diagonal with entries given by $\LS$, and $J_{\mathrm{fast}}$ consists of the Jordan blocks corresponding to the fast eigenvalues. For such $U$, define
\begin{equation*}
    K_U(t) \colonequals S \diag\brac{0, \brac{e^{it\Im(\lambda)}}_{\lambda \in \LS}, 0} S^{-1}
\end{equation*}
for $t \in \R$, using the same ordering of the slow eigenvalues and with zero blocks corresponding to the eigenvalue $1$ and the fast eigenvalues. The matrix $K_U(t)$ is real and does not depend on the choice of Jordan basis. We write $\fnorm{\cdot}$ for the Frobenius norm.

\begin{thm}
\label{thm:main}
    Suppose that $U$ satisfies the following conditions:
    \begin{enumerate}
        \item Every slow eigenvalue has equal algebraic and geometric multiplicities.
        \item The function $t \mapsto \fnorm{K_U(t)}$ is constant on $\R$, with value denoted by $\kappa_U$.
    \end{enumerate}
    Then, for every $\theta \in \R$, we have
    \begin{equation*}
        d_n\brac{\frac{1}{2\gamma} \log n + \theta} \xrightarrow[n \to \infty]{} 2\Phi \brac{\frac{\kappa_U \sqrt{d}}{2} e^{-\gamma\theta}} - 1,
    \end{equation*}
    where $\Phi$ is the cumulative distribution function of a standard normal random variable.
\end{thm}

Theorem~\ref{thm:main} provides a partial answer to Question~1.9 of Goenka, Hermon, and Schmid~\cite{GHS} by identifying an explicit Gaussian cutoff profile under the above spectral assumptions. Condition~1 holds on an open dense subset of full Lebesgue measure in the simplex of probability measures on $S_d$. This follows by observing that the discriminant of the characteristic polynomial of $U$ is a polynomial in the weights of $\mu$ that is not identically zero on the simplex, and that $U$ has distinct eigenvalues outside its zero set. A short argument based on Lemma~\ref{lem:subseq-profile} also shows that, under condition~1, condition~2 is necessary for the existence of a limit profile. In Section~\ref{subsec:examples}, we give an example satisfying condition~1 for which no limit profile exists. On the other hand, we show that Theorem~\ref{thm:main} applies to all reversible chains and derive an explicit cutoff profile for the cyclic shift chain studied in~\cite{GHS}.

Our next result determines the mixing time up to an additive constant for every irreducible $\mu$, without either spectral assumption of Theorem~\ref{thm:main}. When condition~1 fails, the mixing time includes an additional term of order $\log\log n$. Let $r+1$ be the largest size of a Jordan block of $U$ corresponding to a slow eigenvalue.

\begin{prop}
\label{prop:mixing-time}
    For every fixed $\eps \in (0, 1)$, we have
    \begin{equation*}
        \tmix(\eps) = \frac{\log n}{2\gamma} + \frac{r}{\gamma} \log \log n + O_{\eps}(1).
    \end{equation*}
    In particular, the sequence $\{\bX^{(n)}\}_{n \ge 2}$ exhibits total variation cutoff with an $O(1)$ window.
\end{prop}

By~\cite{GHS}*{Lemma~1.6}, both Theorem~\ref{thm:main} and Proposition~\ref{prop:mixing-time} also hold for the labeled version of our chains, in which the balls are all distinguishable. When the urn chains in the sequence $\{\bX^{(n)}\}_{n \ge 2}$ are reversible, $U$ is symmetric by~\cite{GHS}*{Lemma~3.4}, so condition~1 of Theorem~\ref{thm:main} holds and hence $r = 0$. In this case, the estimate $\tmix(\eps) = (\log n)/(2\gamma) + O_\eps(1)$ already follows from \cite{GHS}. In the nonreversible case, their general spectral estimate leaves an additive $O_\eps(\log\log n)$ error in the mixing time. Proposition~\ref{prop:mixing-time} identifies the precise $\log\log n$ correction and bounds the remainder by $O_\eps(1)$. In Section~\ref{subsec:examples}, we construct a family of examples showing that this correction is necessary in general.

\subsection{Examples}
\label{subsec:examples}

Both conditions of Theorem~\ref{thm:main} hold whenever $U$ is normal. Indeed, $U$ can then be diagonalized by a unitary matrix, so condition~1 holds. By unitary invariance of the Frobenius norm,
\begin{equation*}
    \fnorm{K_U(t)}^2 = \sum_{\lambda \in \LS} \abs{e^{it\Im(\lambda)}}^2 = \abs{\LS}
\end{equation*}
for all $t \in \R$. Thus, condition~2 holds with $\kappa_U = \sqrt{\abs{\LS}}$, where eigenvalues are counted with algebraic multiplicity. We remark that normality is not necessary for the existence of a limit profile. Indeed, if condition~1 holds and all slow eigenvalues are real, then $K_U(t)$ is constant in $t$, so condition~2 also holds.\vspace{3pt}

\textbf{Reversible chains.} By~\cite{GHS}*{Lemma~3.4}, the balanced urn chain is reversible if and only if $\mu(\sigma) = \mu(\sigma^{-1})$ for all $\sigma \in S_d$. In this case, $U$ is symmetric, so both conditions of Theorem~\ref{thm:main} hold. Note that this also includes the chains arising from conjugacy-invariant measures.

In particular, consider the \emph{mean-field chain}, in which $\mu$ is uniform on the transpositions in $S_d$. The only nontrivial eigenvalue of $U$ is $1-2/(d-1)$, with multiplicity $d-1$, so
\begin{equation*}
    \gamma = \frac{2}{d-1} \quad \text{ and } \quad \kappa_U = \sqrt{d-1}.
\end{equation*}
Theorem~\ref{thm:main} therefore gives, for every $\theta \in \R$,
\begin{equation*}
    d_n \brac{\frac{d-1}{4}\log n+\theta} \xrightarrow[n\to\infty]{} 2\Phi \brac{\frac{\sqrt{d(d-1)}}{2} e^{-2\theta/(d-1)}} - 1.
\end{equation*}
This profile was previously obtained by Shen~\cite{Shen} using Fourier-analytic methods and combinatorial central limit theorems.\vspace{3pt}

\textbf{The cyclic shift chain.} Suppose $d \ge 3$ and $\mu$ is concentrated on the cycle $(1\,2\,\dots\,d)$. The urn chain is nonreversible, but $U$ is a permutation matrix and hence normal. Its eigenvalues are the $d$th
roots of unity, and the two slow eigenvalues are
$e^{2\pi i/d}$ and $e^{-2\pi i/d}$. Consequently,
\begin{equation*}
    \gamma = 1 - \cos\brac{\frac{2\pi}{d}} \quad \text{ and } \quad \kappa_U = \sqrt{2}.
\end{equation*}
Theorem~\ref{thm:main} therefore gives, for every $\theta \in \R$,
\begin{equation*}
    d_n\brac{\frac{\log n}{4 \sin^2(\pi/d)}+\theta} \xrightarrow[n\to\infty]{} 2\Phi \brac{\sqrt{\frac d2} e^{-2 \theta \sin^2(\pi/d)}} - 1.
\end{equation*}
Thus, non-real slow eigenvalues do not necessarily prevent the existence of a cutoff profile.\vspace{3pt}

\textbf{Nonexistence of a limit profile.}
Let $d = 3$ and let $\mu$ assign probabilities $2/3$ and $1/3$ to $(123)$ and $(12)$, respectively. The corresponding single-ball
transition matrix is
\begin{equation*}
    U = \frac{1}{3} \begin{pmatrix}
        0&3&0\\
        1&0&2\\
        2&0&1
    \end{pmatrix},
\end{equation*}
whose eigenvalues are $1$ and $-1/3 \pm i\sqrt{2}/3$. Thus, condition~1 holds, $\gamma = 4/3$, and $r = 0$. However, a direct calculation gives, for $t \in \R$,
\begin{equation*}
    \fnorm{K_U(t)}^2 = 2 + 2\sin^2\brac{\frac{\sqrt{2}}{3}t},
\end{equation*}
so condition~2 fails. At time $t = (3/8)\log n$, the phase in the above expression is $(\sqrt{2}/8)\log n$, which approaches both $0$ and $\pi/2$ modulo $\pi$ along suitable subsequences. Lemma~\ref{lem:subseq-profile} therefore implies that $d_n((3/8)\log n)$ does not converge. Nevertheless,
Proposition~\ref{prop:mixing-time} gives
\begin{equation*}
    \tmix(\eps)
    =\frac38\log n+O_\eps(1)
\end{equation*}
for every fixed $\eps\in(0,1)$.
Hence, the chain exhibits cutoff with an $O(1)$ window and no
$\log\log n$ correction, but has no limit profile at this centering.\vspace{3pt}

\textbf{Necessity of the $\log \log n$ term.} Fix an integer $k \ge 2$, set $d = 2^k$, and identify the urns with the binary strings in $\{0,1\}^k$. Let $\mu$ be uniform on the two permutations
\begin{align*}
    \sigma_0(x_1,\dots,x_k)
    &= (x_2,\dots,x_k,x_1),\\
    \sigma_1(x_1,\dots,x_k)
    &= (x_2,\dots,x_k,1-x_1).
\end{align*}
At each update, the single-ball chain therefore shifts the string one place to the left and appends an independent uniform bit. After $k$ steps, all bits have been refreshed, whereas after $k-1$ steps the first bit is still the initial last bit. Hence, $U^k = d^{-1}\ind\ind^\top$, while $U^{k-1} \ne d^{-1}\ind\ind^\top$. Consequently, $U$ is irreducible and its restriction to the subspace of vectors with zero coordinate sum is nilpotent of index $k$. Thus, all nontrivial eigenvalues are zero, and the largest Jordan block corresponding to these eigenvalues has size $k$. Consequently, $\gamma = 1$ and $r = k-1$, and so Proposition~\ref{prop:mixing-time} gives, for every fixed $\eps \in (0,1)$,
\begin{equation*}
    \tmix(\eps) = \frac{1}{2} \log n + (k-1) \log \log n + O_\eps(1).
\end{equation*}
This yields an example with $r$ equal to any prescribed positive integer.

\subsection{Proof ideas}
\label{subsec:proof-ideas}

Our proof of Theorem~\ref{thm:main} follows the diffusion approximation approach of Olesker-Taylor and Schmid~\cite{OTS}. Write $P_t\colonequals e^{t(U-I)}$ for the single-ball transition kernel, and let $\Pi$ be the $d \times d$ matrix whose entries are all $1/d$. Starting from the worst-case configuration $X^{(n)}(0) = nI$, the mean is $\E_{nI} X^{(n)}(t) = nP_t$ and $\E_{nI} \fnorm{X^{(n)}(t)-nP_t}^2 = O(n)$ uniformly in time. The distance between the mean configuration $nP_t$ and the stationary mean $n\Pi$ is $n \fnorm{P_t-\Pi}$. At times $t = t(n)$ satisfying $n\fnorm{P_t-\Pi} = \omega(\sqrt{n})$, the second-moment bound shows that fluctuations around $nP_t$ are negligible relative to this distance with high probability.

The remaining task is to understand the chain when the distance between the mean configuration and stationary mean becomes comparable to the $\sqrt{n}$ scale of the fluctuations. Under condition~1, the Jordan decomposition gives, at times $t = (\log n)/(2\gamma) + \theta$ with fixed $\theta \in \R$,
\begin{equation*}
    \fnorm{n(P_t - \Pi) - \sqrt{n} e^{-\gamma\theta} K_U(t)} = o(\sqrt{n}).
\end{equation*}
This describes the mean displacement from the stationary mean with an error smaller than the fluctuation scale, but identifying the cutoff profile also requires a description of the distribution around the mean. To obtain this, we center the chain at $n\Pi$, scale by $\sqrt{n}$, and approximate the rescaled dynamics near equilibrium by an Ornstein--Uhlenbeck process. The limiting diffusion has a centered Gaussian stationary distribution with covariance $d^{-1}$ times the identity on the space of real $d \times d$ matrices with zero row and column sums. The rescaled stationary distributions of the urn chain also converge to this same Gaussian law.

To identify the limit profile, we start the diffusion approximation at time $(\log n)/(2\gamma) - C$, with $C > 0$ large and fixed. Since the rescaled means may oscillate, we first work along subsequences of $n$ on which they converge. The weak convergence results discussed in the previous paragraph give the lower bound on total variation, but the upper bound requires an additional coupling step. We adapt the argument of~\cite{BBL}*{Proposition~3.10} to show that two copies started near equilibrium at distance $o(\sqrt{n})$ can be coupled to agree with high probability after any fixed time $\delta > 0$. After passing to the limit along such a subsequence, we let $C\to\infty$ so that the covariance of the approximating Gaussian converges to that of the stationary Gaussian, and let $\delta \to 0^+$ to remove the additional coupling time. This yields matching bounds in terms of the total variation distance between two Gaussians with the same isotropic covariance. This distance depends only on the Frobenius norm of their mean difference, so condition~2 makes the subsequential limits agree and yields the claimed profile, even when the mean displacement itself oscillates.

For Proposition~\ref{prop:mixing-time}, neither spectral assumption is needed. The Jordan decomposition gives
\begin{equation*}
    \fnorm{P_t-\Pi} = \Theta(t^r e^{-\gamma t})
\end{equation*}
as $t \to \infty$, where the factor $t^r$ comes from the largest slow Jordan blocks. The estimates of Goenka, Hermon, and Schmid~\cite{GHS}, together with equivalence of norms in fixed dimension, locate the mixing time within $O_\eps(1)$ of the first time the mean configuration $nP_t$ comes within a fixed multiple of $\sqrt{n}$ of the stationary mean $n\Pi$ in Frobenius norm. At this time, $\sqrt{n}\,t^r e^{-\gamma t} = \Theta(1)$. Solving for $t$ gives the $(r/\gamma) \log \log n$ correction, as shown in Lemma~\ref{lem:spectral-scale}.

\subsubsection*{Organization}

The remainder of the paper is organized as follows. In Section~\ref{sec:properties}, we establish basic properties of the chain, prove concentration around its deterministic mean trajectory, and obtain a precise mixing time estimate. In Section~\ref{sec:diffusion}, we establish convergence of the rescaled chain to an Ornstein--Uhlenbeck diffusion. Finally, in Section~\ref{sec:profile}, we combine these results with a local coupling argument to prove our main theorem.

\subsubsection*{Acknowledgments}

The authors thank Omer Angel, Jonathan Hermon, Dominik Schmid, and Allan Sly for helpful discussions. RG is supported by a joint Clarendon and Exeter College SKP scholarship.

\subsubsection*{Statement on AI use}

The strategy used to show the main result is entirely due to the authors. A conversation with ChatGPT (GPT-5.6 Sol) helped us find \cite{BBL}, and more specifically its Proposition~3.10, an important ingredient to complete our argument, as we explain in Section~\ref{sec:profile}. ChatGPT (GPT-6 Astra Pro) was also used to help polish the exposition of this paper.

\section{Dynamics and mixing time}
\label{sec:properties}

\subsection{Basic properties}
\label{subsec:basic-prop}

We first show that the configuration $nI$, in which urn $i$ contains all balls of color $i$ for every $i \in [d]$, maximizes the total variation distance from stationarity at every time. Consequently, to determine the limit profile, it suffices to consider the chain started from $nI$.

\begin{lem}
\label{lem:worst-start}
    For every $t \ge 0$, we have
    \begin{equation*}
        d_n(t) = \tv(H_t^{(n)}(nI, \cdot), \pi_n).
    \end{equation*}
\end{lem}

\begin{proof}
    Consider the labeled version of the chain, in which all $dn$ balls are distinguishable (see~\cite{GHS}*{Section~1.3} for a formal definition). Denote its state space by $\widetilde{\Omega}_n$, its transition kernel by $\widetilde{H}_t^{(n)}$, and its stationary distribution by $\widetilde{\pi}_n$. The latter is uniform on $\widetilde{\Omega}_n$.

    Fix an arbitrary labeled configuration $\eta \in \widetilde{\Omega}_n$, and for $j \in [d]$, let $B_j \colonequals B_j(\eta)$ be the set of $n$ balls in urn $j$ in the configuration $\eta$. For any coloring $c$ of the balls with exactly $n$ balls of each color, define the projection $\varphi_c \colon \widetilde{\Omega}_n \to \Omega_n$ by
    \begin{equation*}
        \varphi_c(x)_{ij} \colonequals \#\set{b}{c(b) = i,\ b \text{ is in urn } j \text{ in configuration } x}.
    \end{equation*}
    The projected process has the dynamics of the balanced chain. By considering the coloring $c_0$ which assigns color $j$ to every ball in $B_j$ for each $j \in [d]$, one can show that
    \begin{equation}
    \label{eq:labelled-tv}
        \tv\bigl(\widetilde{H}_t^{(n)}(\eta,\cdot), \widetilde{\pi}_n\bigr) = \tv\bigl(H_t^{(n)}(nI,\cdot),\pi_n\bigr).
    \end{equation}
    See the proof of \cite{GHS}*{Lemma~1.6} for a detailed justification of \eqref{eq:labelled-tv}. Now fix any $x \in \Omega_n$. Choose a coloring $c$ such that, for each $i, j \in [d]$, exactly $x_{ij}$ balls in $B_j$ receive color $i$. Such a coloring exists because the row and column sums of $x$ are all equal to $n$. By construction, $\varphi_c(\eta) = x$. Since total variation distance cannot increase under projections, we obtain
    \begin{equation*}
        \tv\bigl(H_t^{(n)}(x, \cdot), \pi_n\bigr) \le \tv\bigl(\widetilde{H}_t^{(n)}(\eta, \cdot), \widetilde{\pi}_n\bigr) = \tv\bigl(H_t^{(n)}(nI,\cdot), \pi_n\bigr),
    \end{equation*}
    where the last equality follows from \eqref{eq:labelled-tv}. Taking supremum over $x \in \Omega_n$ proves the claim.
\end{proof}

Let $e_1, \dots, e_d$ denote the standard basis column vectors of $\R^d$, and set $E_{i,j} \colonequals e_i e_j^\top$. Let $\cJ$ denote the set of possible jumps of the chain at an update. Every $J \in \cJ$ is of the form
\begin{equation}
\label{jump-form}
    J = \sum_{j=1}^d (E_{i_j,\sigma(j)} - E_{i_j,j}),
\end{equation}
where $\sigma \in \supp(\mu)$ and for each urn $j$, $i_j$ is the chosen color from that urn. This implies that for $J \in \cJ$, we have $\onenorm{J} \le 2d$ and $\fnorm{J} \le \sqrt{2d}$, where $\onenorm{\cdot}$ denotes the entrywise $\ell^1$ norm. Write $q_n(x,y)$ for the generator entries of $\bX^{(n)}$. For $x \in \Omega_n$ and $J \in \cJ \setminus \{0\}$, we have
\begin{equation}
\label{eqn:trans}
    \frac{1}{n} q_n(x,x+J) = \sum_{\substack{\sigma,i_1,\dots,i_d\\ J\ \text{satisfies}\ \eqref{jump-form}}} \mu(\sigma) \prod_{j=1}^d \frac{x_{i_j,j}}{n}.
\end{equation}

In Section~\ref{sec:profile}, we will need this formula for all nonnegative real matrices $x$ whose row and column sums are all equal to $n$. Fix $0 < \delta_0 < d^{-1}$. For all sufficiently large $n$, if $\fnorm{x-n\Pi}\le\delta_0n$, then $x+J$ is nonnegative for every $J \in \cJ$ and has the same row and column sums as $x$. Moreover, for every $J\in\cJ\setminus\{0\}$, we have
\begin{equation}
\label{rate-blue}
    \frac{1}{n}q_n(x,x+J) \ge \brac{\min_{\sigma \in \supp(\mu)}\mu(\sigma)} \brac{d^{-1}-\delta_0}^d \colonequals r_0.
\end{equation}

We next show that the graph distance between two configurations is bounded above by a constant multiple of the Frobenius norm of their difference.

\begin{lem}
\label{alg-dist}
    Let $\dis_n$ denote the graph distance on $\Omega_n$, where distinct configurations $x, y$ are adjacent if and only if $q_n(x, y) > 0$ or $q_n(y, x) > 0$. There exists a constant $c > 0$ such that
    \begin{equation*}
        \dis_n(x, y) \le c \fnorm{x-y}
    \end{equation*}
    for all $x, y \in \Omega_n$.
\end{lem}

\begin{proof}
    We first show that it is possible to swap the positions of any two balls and return every other ball to its original urn, in at most $c_d = \frac{(2d)!}{2^d} - 1$ moves. Choose two balls from each urn, including the two balls we want to swap, and consider updates involving only these selected balls. Due to Lemma~3.1 of \cite{GHS}, the chain in which these balls are labeled is irreducible, and has $\frac{(2d)!}{2^d}$ states. Therefore, its diameter is at most $c_d$, which means that we can swap our two balls in at most $c_d$ moves.

    Now let $x, y \in \Omega_n$ with $x \ne y$. Choose a color $i$ and an urn $a$ such that $x_{i,a} > y_{i,a}$. Since the row sums of $x$ and $y$ agree, there is an urn $b$ such that $x_{i,b} < y_{i,b}$. Since their column sums also agree, there is a color $j$ such that $x_{j,b} > y_{j,b}$. Start with configuration $x$ and exchange a ball of color $i$ in urn $a$ with a ball of color $j$ in urn $b$, and denote the resulting configuration by $x'$. Then, on replacing $x$ with $x'$, the absolute errors $\abs{x_{i,a} - y_{i,a}}$, $\abs{x_{i,b} - y_{i,b}}$, and $\abs{x_{j,b} - y_{j,b}}$ each decrease by one, while $\abs{x_{j,a} - y_{j,a}}$ increases by at most one. Thus, we have $\onenorm{x'-y} \le \onenorm{x-y} - 2$. Repeating this procedure reaches $y$ after at most $\onenorm{x-y}/2$ exchanges. Since each exchange can be realized in at most $c_d$ updates, we obtain
    \begin{equation*}
        \dis_n(x,y) \le \frac{c_d}{2} \onenorm{x-y} \le \frac{d c_d}{2} \fnorm{x-y}.
    \end{equation*}
    Setting $c \colonequals d c_d$ finishes the proof of the lemma.
\end{proof}

We will also use the following bound on the probability of leaving a neighborhood of the stationary mean $n\Pi$ over a short time interval.

\begin{lem}
\label{stay-close-center}
    Let $\delta_0 > 0$, and suppose that
    $\fnorm{X^{(n)}(0) - n\Pi} \le \delta_0 n/2$. There exists a positive constant $c = c(\delta_0)$ such that, for every $0 < \eps < \delta_0/(4\sqrt{2d})$, we have
    \begin{equation*}
        \P \brac{\sup_{0 \le t \le \eps} \fnorm{X^{(n)}(t)-n\Pi} > \delta_0 n} \le e^{-cn}.
    \end{equation*}
\end{lem}

\begin{proof}
    By a Poisson tail bound, there is a constant $c_d > 0$ such that with probability at least $1 - \exp(-c_d \delta_0 n)$, there are at most $\frac{\delta_0}{2\sqrt{2d}}n$ updates, each of which can change the Frobenius norm by at most $\sqrt{2d}$. Therefore, with probability at least $1 - \exp(-c_d \delta_0 n)$, the Frobenius norm changes by at most $\delta_0 n/2$, which implies the lemma.
\end{proof}

\subsection{Concentration around a deterministic trajectory}

Throughout this subsection, we take $X^{(n)}(0) = nI$. We determine the mean of $X^{(n)}(t)$ and establish a second-moment bound that is uniform in time. Our argument uses the row processes $(X_{i, \cdot}^{(n)}(t))_{t \ge 0}$, each of which is a Markov chain on the nonnegative integer row vectors whose coordinates sum to $n$. For $ \theta \in \R$, define
\begin{equation*}
    t_{n,\theta} \colonequals \frac{\log n}{2\gamma} + \theta.
\end{equation*}
We will show that, for sufficiently large $C > 0$, the fluctuations of $X^{(n)}(t_{n,-C})$ around its mean are small relative to the distance between the mean and $n \Pi$.

\begin{prop}
\label{prop:concentration}
    For each $i \in [d]$, the process $(X_{i,\cdot}^{(n)}(t))_{t \ge 0}$ has mean
    \begin{equation*}
        m_i(t) \colonequals \E\sq{X_{i,\cdot}^{(n)}(t)} = ne_i^\top \exp(t(U-I))
    \end{equation*}
    and its centered second moment satisfies
    \begin{equation*}
        \E\twonorm{X_{i,\cdot}^{(n)}(t)-m_i(t)}^2 = O(n)
    \end{equation*}
    uniformly in $t \ge 0$. Consequently, for every $\eps > 0$, if $C > 0$ is sufficiently large, then
    \begin{equation*}
        \P\brac{\fnorm{X^{(n)}(t_{n,-C})-m(t_{n,-C})} \ge \eps e^{\gamma C} \sqrt{n}/2}< \eps
    \end{equation*}
    for all sufficiently large $n$, where $m(t)$ is the matrix whose rows are $m_i(t)$.
\end{prop}

The second moment bound in Proposition~\ref{prop:concentration} also follows from \cite{GHS}*{Lemma~4.7}. We give a direct proof in the present setting, using the following standard identity for finite-state continuous-time Markov chains.

\begin{prop}
\label{prop:ODE}
    Let $(W_t)_{t \ge 0}$ be a continuous-time Markov chain with finite state space $\cW$ and transition rate matrix $Q$. Then, for any function $f \colon \cW \to \R^k$, we have
    \begin{equation*}
        \frac{\d}{\d t} \E \sq{f(W_t)} = \E \sq{(Qf)(W_t)},
    \end{equation*}
    where $(Qf)(x) = \sum_{y \in \cW} q(x,y) f(y)$.
\end{prop}

\begin{proof}[Proof of Proposition~\ref{prop:concentration}]
    Fix $i \in [d]$, and let $Q = (q(x,y))$ denote the generator of the row process $(X_{i, \cdot}^{(n)}(t))_{t \ge 0}$. Write $x = (x_1,\dots,x_d)$ for a state of this process. At an update, the expected change in coordinate $x_j$ is
    \begin{equation*}
        \sum_{s=1}^d \frac{x_s}{n} U_{s,j} - \frac{x_j}{n}.
    \end{equation*}
    Since updates occur at rate $n$, for the function $f_j^{(1)}(x) = x_j$, we obtain
    \begin{equation*}
        (Qf_j^{(1)})(x) = \sum_{s=1}^d U_{s,j} x_s - x_j = (x(U-I))_j.
    \end{equation*}
    By Proposition~\ref{prop:ODE}, we get
    \begin{equation*}
        m_i'(t) = m_i(t)(U-I), \qquad m_i(0) = ne_i^\top.
    \end{equation*}
    Hence, $m_i(t) = ne_i^\top P_t$, thereby proving the first claim.

    For every real row vector $z$, double stochasticity of $U$ gives
    \begin{equation*}
        \inner{z}{z(U-I)} = -\frac{1}{2} \sum_{j,k} U_{j,k} (z_j-z_k)^2.
    \end{equation*}
    By irreducibility, the right-hand side vanishes only when $z$ is constant. Compactness of the unit sphere in the zero-sum subspace
    therefore gives a constant $\alpha > 0$ such that
    \begin{equation*}
        \inner{z}{z(U-I)} \le -\alpha \twonorm{z}^2
    \end{equation*}
    whenever $\sum_{j=1}^d z_j = 0$.
    
    To prove the second claim, we apply the generator to the function $f_j^{(2)}(x) = x_j^2$. Since $y_j^2-x_j^2=2x_j(y_j-x_j)+(y_j-x_j)^2$, we get that
    \begin{align*}
        (Qf_j^{(2)})(x) &= \sum_{y \neq x} q(x,y)(y_j^2-x_j^2)\\
        &= 2 x_j (Qf_j^{(1)})(x) + \sum_{y \ne x} q(x,y) (y_j - x_j)^2\\
        &\le 2 x_j \brac{x(U-I)}_j + n,
    \end{align*}
    since each coordinate changes by at most one and the total jump rate of the row process is at most $n$. Proposition~\ref{prop:ODE}, together with the above inequality, implies that the function
    \begin{equation*}
        V_i(t)\colonequals\E\twonorm{X_{i,\cdot}^{(n)}(t)-m_i(t)}^2 = \E\sq{\sum_{j=1}^dX_{i,j}^{(n)}(t)^2} - \twonorm{m_i(t)}^2
    \end{equation*}
    satisfies
    \begin{align*}
        \frac{\d}{\d t} V_i(t) &= \sum_{j=1}^d \E\sq{(Qf_j^{(2)})(X_{i,\cdot}^{(n)}(t))} -\sum_{j=1}^d 2m_{i,j}(t) m_{i,j}'(t)\\
        &\le 2\E\inner{X_{i,\cdot}^{(n)}(t)}{X_{i,\cdot}^{(n)}(t)(U-I)} + dn - 2\inner{m_i(t)}{m_i(t)(U-I)}\\
        &= 2\E\inner{X_{i,\cdot}^{(n)}(t)-m_i(t)}{\brac{X_{i,\cdot}^{(n)}(t)-m_i(t)} (U-I)} + dn\\&
        \le -2\alpha V_i(t) + dn.
    \end{align*}
    Since $V_i(0) = 0$, solving the differential inequality implies that $V_i(t) \le \frac{dn}{2\alpha}$ for every $t > 0$.

    For the final claim, summing the second-moment bounds over $i \in [d]$ and applying Chebyshev's inequality gives
    \begin{equation*}
        \P\brac{\fnorm{X^{(n)}(t_{n,-C})-m(t_{n,-C})} \ge \eps e^{\gamma C}\sqrt{n}/2} \le \frac{2d^2}{\alpha\eps^2} e^{-2\gamma C}
    \end{equation*}
    for all sufficiently large $n$. The claim follows when $C$ is large enough.
\end{proof}

\subsection{Precise estimate for the mixing time}
\label{subsec:mixing-time}

We now prove Proposition~\ref{prop:mixing-time}. We first use the Jordan decomposition to determine, up to an additive constant, the time at which $P_t-\Pi$ reaches a fixed multiple of $n^{-1/2}$ in any fixed matrix norm.

\begin{lem}
\label{lem:spectral-scale}
    Let $\snorm{\cdot}$ be a fixed norm on $\R^{d \times d}$. For each fixed $a > 0$, define
    \begin{equation*}
        \tau_n(a) \colonequals \inf \set{t \ge 0}{\snorm{P_t - \Pi} \le a n^{-1/2}}.
    \end{equation*}
    Then, as $n \to \infty$, we have
    \begin{equation*}
        \tau_n(a) = \frac{\log n}{2\gamma} + \frac{r}{\gamma} \log \log n + O_{\snorm{\cdot},a}(1).
    \end{equation*}
\end{lem}

\begin{proof}
    By the Perron--Frobenius theorem, $1$ is a simple eigenvalue of $U$ and $\gamma > 0$. Further, by the Jordan decomposition of $U$, we may choose an invertible complex matrix $S$ such that
    \begin{equation*}
        S^{-1} U S = \diag(1, J_1, \dots, J_q),
    \end{equation*}
    where $J_j = \lambda_j I + N_j$ is a Jordan block of size $m_j$, with $\lambda_j \neq 1$, for $j \in [q]$. Then we have
    \begin{equation}
    \label{eqn:exp}
        P_t - \Pi = S \diag\bigl(1, e^{t(J_1-I)}, \dots, e^{t(J_q-I)}\bigr) S^{-1} - \Pi.
    \end{equation}
    Since $U$ is doubly stochastic, every generalized eigenvector corresponding to an eigenvalue other than $1$ has coordinate sum zero. Thus, choosing the first column of $S$ to be $\ind$, we have
    \begin{equation*}
        \Pi S = S \diag(1, 0, \dots, 0).
    \end{equation*}
    Multiplying both sides by $S^{-1}$ on the right, we obtain $\Pi = S \diag(1, 0, \dots, 0) S^{-1}$. Using this in \eqref{eqn:exp}, we obtain $P_t - \Pi = S B_t S^{-1}$, where $B_t \colonequals \diag\bigl(0, e^{t(J_1-I)}, \dots, e^{t(J_q-I)}\bigr)$.
    
    Equivalence of norms on $\R^{d \times d}$ and the fact that $S$ is fixed imply
    \begin{equation}
    \label{eqn:equiv}
        \snorm{P_t - \Pi} = \Theta_{\snorm{\cdot}}(\fnorm{B_t}).
    \end{equation}
    So, we obtain precise estimates for $\fnorm{B_t}$ instead. For each $j$, $N_j$ is an $m_j \times m_j$ matrix with ones on the diagonal above the main diagonal and zeros elsewhere. Therefore, we have
    \begin{equation*}
        e^{t(J_j-I)} = e^{t(\lambda_j - 1)I} e^{tN_j} = e^{(\lambda_j-1)t} e^{tN_j} = e^{(\lambda_j-1)t} \sum_{k=0}^{m_j-1}\frac{t^kN_j^k}{k!}.
    \end{equation*}
    For $k \in \{0, \dots, m_j - 1\}$, the matrices $N_j^k$ have disjoint supports and $\fnorm{N_j^k}^2 = m_j-k$. Consequently, we have
    \begin{equation*}
        \fnorm{B_t}^2 = \sum_{j=1}^{q} e^{-2(1-\Re(\lambda_j))t} \sum_{k=0}^{m_j-1} \frac{m_j-k}{(k!)^2}t^{2k}.
    \end{equation*}
    For $t$ large, the dominant terms are those with $\Re(\lambda_j) = 1-\gamma$ and $m_j = r+1$. Combining this observation with \eqref{eqn:equiv}, it follows that there exist constants $c_1, c_2 > 0$ such that
    \begin{equation}
    \label{eq:spectral-two-sided}
        c_1 t^re^{-\gamma t} \le \snorm{P_t - \Pi} \le c_2 t^re^{-\gamma t}
    \end{equation}
    for all sufficiently large $t$. This implies $\snorm{P_t - \Pi} \to 0$ as $t \to \infty$, so $\tau_n(a) < \infty$. Note that $P_t - \Pi$ is nonzero for every finite $t$ since it has a nonzero eigenvalue $e^{t(\lambda - 1)}$ for each $\lambda \in \Lambda \setminus \{1\}$. Continuity of $\snorm{P_t - \Pi}$ therefore implies that $\tau_n(a) \to \infty$ as $n \to \infty$. Moreover, for all sufficiently large $n$, we have
    \begin{equation*}
        \snorm{P_{\tau_n(a)} - \Pi} = a n^{-1/2}.
    \end{equation*}
    Taking logarithms in \eqref{eq:spectral-two-sided} at $t = \tau_n(a)$ gives
    \begin{equation}
    \label{eqn:tau}
        \gamma \tau_n(a) = \frac{1}{2} \log n + r \log \tau_n(a) + O(1).
    \end{equation}
    Dividing by $\tau_n(a)$ first yields $\tau_n(a) \sim (\log n)/(2\gamma)$, which implies
    \begin{equation*}
        \log \tau_n(a) = \log \log n - \log(2\gamma) + o(1).
    \end{equation*}
    Substituting this into \eqref{eqn:tau} proves the assertion.
\end{proof}

\begin{proof}[Proof of Proposition~\ref{prop:mixing-time}]
    Let $\nu$ be the uniform distribution on $[d]$, and write
    \begin{equation*}
        d_{\mathrm{single}}(t) \colonequals \max_{i\in[d]} \tv \bigl(P_t(i,\cdot), \nu\bigr).
    \end{equation*}
    For $M \in \R^{d\times d}$, define the norm
    \begin{equation*}
        \rnorm{M} \colonequals \frac{1}{2} \max_{i \in [d]} \sum_{j=1}^d \abs{M_{ij}}.
    \end{equation*}
    Then $d_{\mathrm{single}}(t) = \rnorm{P_t - \Pi}$. For $n \in \N$, set
    \begin{equation*}
        s_n \colonequals \inf \set{t \ge 0}{d_{\mathrm{single}}(t) \le \sqrt{\frac{d}{n}}}.
    \end{equation*}
    Applying Lemma~\ref{lem:spectral-scale} with the row norm and $a = \sqrt{d}$ gives
    \begin{equation}
    \label{eq:common-spectral-scale}
        s_n = \frac{\log n}{2\gamma} + \frac{r}{\gamma} \log \log n + O(1).
    \end{equation}
    It therefore suffices to prove $\tmix(\eps) = s_n + O(1)$.
    
    We will now use the center estimates from~\cite{GHS}. For $R > 0$, define the center
    \begin{equation*}
        \cC_n(R) \colonequals \set{x \in \Omega_n}{x_{ij} \ge \frac{n}{d} - R \sqrt{\frac{n}{d}} \text{ for every } i, j \in [d]}.
    \end{equation*}
    In applying the estimates from \cite{GHS}, we take $m = d$ and use $n/d$ in place of their parameter $n$. Since our updates occur at rate $n$, their rate-one time parameters are divided by $n$. The balanced formulation discussed in \cite{GHS}*{Section~1.2} removes any divisibility restriction on $n$. Fix $\eps \in (0,1)$ and choose $\delta \in (0, \min \{\eps, 1 - \eps\}/2)$. Lemma~4.8 of \cite{GHS} gives constants $R, \alpha > 0$ and states $x_n \in \Omega_n$ such that, for all sufficiently large $n$, we have $\pi_n\brac{\cC_n(R)} \ge 1-\delta$, and
    \begin{align*}
        \max_{x \in \Omega_n} \P_x\!\brac{X^{(n)}(s_n+\alpha) \notin \cC_n(R)} &\le \delta, \\
        \P_{x_n}\!\brac{X^{(n)}(s_n-\alpha) \in \cC_n(R)} &\le \delta.
    \end{align*}
    Also, \cite{GHS}*{Section~10.1, equation~(86)} gives a constant $T = T(d, \mu, R, \delta) < \infty$ such that
    \begin{equation*}
        \sup_{y \in \cC_n(R)} \tv \brac{H_T^{(n)}(y, \cdot), \pi_n} \le \delta
    \end{equation*}
    for all sufficiently large $n$.
    
    For the upper bound, the Markov property and
    convexity of total variation give
    \begin{equation*}
        d_n(s_n + \alpha + T) \le \max_{x \in \Omega_n} \P_x\!\brac{X^{(n)}(s_n + \alpha) \notin \cC_n(R)} + \sup_{y \in \cC_n(R)} \tv \bigl(H_T^{(n)}(y, \cdot), \pi_n\bigr) \le 2\delta < \eps.
    \end{equation*}
    For the lower bound, testing total variation against the same set $\cC_n(R)$ gives
    \begin{equation*}
        d_n(s_n - \alpha) \ge \pi_n \bigl(\cC_n(R)\bigr) - \P_{x_n}\!\brac{X^{(n)}(s_n - \alpha) \in \cC_n(R)} \ge 1 - 2\delta > \eps.
    \end{equation*}
    Since $d_n$ is nonincreasing, we conclude that
    \begin{equation*}
        s_n - \alpha \le \tmix(\eps) \le s_n + \alpha + T
    \end{equation*}
    for all sufficiently large $n$. Combining this with
    \eqref{eq:common-spectral-scale} proves the
    mixing-time estimate. Since $s_n\to\infty$ and the
    estimate holds for every fixed $\eps\in(0,1)$,
    the asserted cutoff follows.
\end{proof}

\section{Convergence to a diffusion}
\label{sec:diffusion}

In this section, we establish a diffusion approximation for the rescaled fluctuations around the stationary mean $n \Pi$. Consider the space
\begin{equation*}
    \cH \colonequals \set{z \in \R^{d \times d}}{z \ind = 0,\ \ind^\top z = 0},
\end{equation*}
equipped with the Frobenius inner product $\finner{z}{w} \colonequals \Tr(z^\top w)$. Define the centered and rescaled process $\bY^{(n)}$ by
\begin{equation*}
    Y^{(n)}(t) \colonequals \frac{X^{(n)}(t)-n\Pi}{\sqrt{n}},
\end{equation*}
for $t \ge 0$, whose state space is given by
\begin{equation*}
    \cH_n \colonequals \set{\frac{x-n\Pi}{\sqrt{n}}}{x\in\Omega_n} \subset \cH.
\end{equation*}

To describe the limiting diffusion, set $L \colonequals 2I - U - U^\top$. Double stochasticity of $U$ gives
\begin{equation*}
    v^\top L v = \sum_{j=1}^d \sum_{k = 1}^d U_{j,k} (v_j - v_k)^2
\end{equation*}
for all $v \in \R^d$. Hence, $L$ is positive semidefinite, and we write $L^{1/2}$ for its unique positive semidefinite square root. Let $W$ be a $d \times d$ matrix of independent standard Brownian motions, and consider the $\cH$-valued Ornstein--Uhlenbeck process
\begin{equation}
\label{eq:sde}
    \d D(t) = D(t) (U - I) \,\d t+\frac1{\sqrt{d}} (I - \Pi) \,\d W(t) L^{1/2}.
\end{equation}
The coefficients are globally Lipschitz and preserve $\cH$, so \eqref{eq:sde} has a unique global solution for each initial state in $\cH$.

We establish convergence to this diffusion by comparing infinitesimal moments. Let $q_n^Y(y, z)$ denote the transition rate of $Y^{(n)}$ from $y$ to $z$. For $y \in \cH_n$ and $F, G \in \cH$, define
\begin{align*}
    b_n(y) &\colonequals \sum_{\substack{z \in \cH_n \\ z \neq y}} q_n^Y(y,z) (z-y),\\
    a_n(y; F,G) &\colonequals \sum_{\substack{z \in \cH_n \\ z \neq y}} q_n^Y(y,z) \finner{z-y}{F} \finner{z-y}{G}.
\end{align*}
Here, $b_n$ is the first infinitesimal moment, while the values of $a_n$ on pairs of vectors in an orthonormal basis of $\cH$ form the second infinitesimal moment matrix.

We apply \cite{Dur}*{Theorem~8.7.1} after identifying $\cH$ isometrically with $\R^{(d-1)^2}$. The infinitesimal moments in that theorem are defined by summing only over jumps of size at most one. Since every jump of $\bY^{(n)}$ has Frobenius norm at most $\sqrt{2d/n}$, these definitions agree with $b_n$ and $a_n$ above for all sufficiently large $n$. Since the coefficients of \eqref{eq:sde} are globally Lipschitz, it suffices to verify that
\begin{equation*}
    b_n(y_n) \to y(U-I) \quad \text{ and } \quad a_n(y_n;F,G) \to \frac{1}{d} \Tr(F^\top GL)
\end{equation*}
whenever $y_n \in \cH_n$ and $y_n \to y \in \cH$, for every $F, G \in \cH$, together with
\begin{equation*}
    \lim_{n \to \infty} \sup_{\substack{y \in \cH_n \\ \fnorm{y} \le R}} \sum_{\substack{z \in \cH_n \\ \fnorm{z-y} > \eps}} q_n^Y(y,z) = 0
\end{equation*}
for every $R, \eps > 0$. The limiting second-moment form is the covariance per unit time of the Brownian term in \eqref{eq:sde}, as we verify in the proof below.

\begin{prop}
\label{prop:diffusion-convergence}
    Suppose that the deterministic initial states $Y^{(n)}(0)\in\cH_n$ converge to $y_0 \in \cH$. Then, for every $T > 0$, we have
    \begin{equation*}
        \brac{Y^{(n)}(t)}_{t\in[0,T]} \Rightarrow (D(t))_{t\in[0,T]},
    \end{equation*}
    in $D([0,T], \cH)$ equipped with the Skorokhod $J_1$ topology, where $D(t)$ solves \eqref{eq:sde} with $D(0) = y_0$. Hence, $D(t) \sim \cN_{\cH} \brac{y_0 P_t, \Sigma_t}$, where $\Sigma_t \colon \cH \to \cH$ is the covariance operator defined by
    \begin{equation*}
        \Sigma_t F = \frac{1}{d} F (I - P_t^\top P_t).
    \end{equation*}
    Moreover, $\opnorm{\Sigma_t-d^{-1}I_{\cH}} = O(t^{2r} e^{-2\gamma t})$ as $t \to \infty$, where $I_\cH$ denotes the identity operator on $\cH$ and the operator norm on $\cH$ is induced by the Frobenius norm.
\end{prop}

\begin{proof}
    Suppose that, at an update, the selected colors are $c = (c_1, \dots, c_d) \in [d]^d$ and the sampled permutation is $\sigma$. The corresponding jump $J \in \cJ$ is
    \begin{equation*}
        J \colonequals \sum_{j=1}^d e_{c_j}(e_{\sigma(j)} - e_j)^\top.
    \end{equation*}
    Then $J\in\cH$, and the jump bound from Section~\ref{subsec:basic-prop} gives $\fnorm{J}^2 \le 2d$.
    At the state $x = n\Pi + y\sqrt{n}$, the colors selected from different urns are independent and
    \begin{equation*}
        \P_x\!\brac{c_j = i} = \frac{x_{ij}}{n} = \frac{1}{d} + \frac{y_{ij}}{\sqrt{n}}
    \end{equation*}
    for all $i, j \in [d]$. Moreover, we have
    \begin{equation*}
        \P_x\!\brac{\sigma(j) = k} = U_{j,k}
    \end{equation*}
    for all $j, k \in [d]$. Consequently, we obtain
    \begin{equation*}
        \E_x\!\sq{J_{ik}} = \sum_{j=1}^d \frac{x_{ij}}{n} U_{j,k} - \frac{x_{ik}}{n} = \sum_{j=1}^d \frac{x_{ij}}n \bigl(U_{j,k} - \ind_{\{j=k\}}\bigr),
    \end{equation*}
    which implies
    \begin{equation*}
        \E_x\!\sq{J} = \frac{x}{n}(U-I) = \frac{y}{\sqrt{n}} (U-I).
    \end{equation*}
    The second equality above follows from the fact that $\Pi(U-I) = 0$. Since $Y^{(n)}$ jumps by $J/\sqrt{n}$ at rate $n$, its first infinitesimal moment is given by
    \begin{equation*}
        b_n(y) = n \E_x\!\sq{J/\sqrt{n}} = y(U - I).
    \end{equation*}
    This precisely agrees with the drift coefficient in \eqref{eq:sde}.
    
    For $F, G \in \cH$, the second infinitesimal moment of $Y^{(n)}$ is given by
    \begin{equation*}
        a_n(y; F,G) = n \E_x\!\sq{\finner{\frac{J}{\sqrt{n}}}{F} \finner{\frac{J}{\sqrt{n}}}{G}} = \E_x\!\sq{\finner{J}{F} \finner{J}{G}}.
    \end{equation*}
    We shall first compute this expectation under the law in which $\sigma \sim \mu$ and the colors $c_1,\dots,c_d$ are i.i.d. uniform on $[d]$, independently of $\sigma$. Denote expectation under this law by $\E_0$. Writing $h_j \colonequals e_{\sigma(j)} - e_j$ and $f_j \colonequals e_{c_j} - d^{-1} \ind$ for $j \in [d]$, we have
    \begin{equation*}
        J = \sum_{j=1}^d e_{c_j} h_j^\top = \sum_{j=1}^d f_j h_j^\top
    \end{equation*}
    since $\sum_{j=1}^d h_j = 0$. Therefore, we have
    \begin{equation*}
        \finner{J}{F} = \Tr(J F^\top) = \Tr\brac{\sum_{j=1}^d f_j h_j^\top F^\top} = \sum_{j=1}^d h_j^\top F^\top f_j.
    \end{equation*}
    Similarly, we have $\finner{J}{G} = \sum_{j=1}^d h_j^\top G^\top f_j = \sum_{j=1}^d f_j^\top G h_j$. Thus, we obtain
    \begin{equation}
    \label{eqn:cov}
        \E_0\!\sq{\finner{J}{F} \finner{J}{G}} = \E_0 \sum_{j=1}^d \sum_{k=1}^d h_j^\top F^\top f_j f_k^\top G h_k.
    \end{equation}
    Conditional on $\sigma$, the vectors $f_j$ are independent and centered, with covariance matrix $d^{-1} (I - \Pi)$. Further, using $\P\brac{\sigma(j)=k} = U_{j,k}$, we obtain
    \begin{equation}
    \label{eqn:hcov}
        \E_0 \sum_{j=1}^d h_j h_j^\top = 2I - \E_0 \sum_{j=1}^d \brac{e_{\sigma(j)}e_j^\top+e_je_{\sigma(j)}^\top} = L.
    \end{equation}
    Using these in \eqref{eqn:cov}, we obtain
    \begin{equation*}
        \E_0\!\sq{\finner{J}{F} \finner{J}{G}} = \E_0 \sum_{j=1}^d h_j^\top F^\top f_j f_j^\top G h_j = \frac{1}{d} \E_0 \sum_{j=1}^d h_j^\top F^\top (I - \Pi) G h_j.
    \end{equation*}
    Since $(I - \Pi)G = G$, we obtain
    \begin{equation*}
        \E_0\!\sq{\finner{J}{F} \finner{J}{G}} = \frac{1}{d} \E_0\!\sq{\Tr\brac{\sum_{j=1}^d F^\top G h_j h_j^\top}} = \frac{1}{d} \Tr\brac{F^\top G L},
    \end{equation*}
    where the last equality uses linearity of expectation and \eqref{eqn:hcov}. Since $(I-\Pi)F = F$ and $(I - \Pi) G = G$, the covariance bilinear form of the diffusion term in \eqref{eq:sde} is
    \begin{equation*}
        \frac{1}{d} \finner{(I - \Pi) F L^{1/2}}{(I - \Pi) G L^{1/2}} = \frac{1}{d} \Tr(F^\top G L).
    \end{equation*}
    So the second moment under the uniform sampling law agrees with that of the limiting diffusion. It remains to compare the actual sampling law with the uniform one. At the state $x$, the probability of choosing
    a given pair $(\sigma, c) \in S_d \times [d]^d$ is
    \begin{equation*}
        \mu(\sigma)\prod_{j=1}^d
        \brac{\frac1d+\frac{y_{c_j,j}}{\sqrt{n}}},
    \end{equation*}
    which converges to $\mu(\sigma)d^{-d}$ uniformly for bounded $y$. Since there are only finitely many pairs $(\sigma,c)$, and $J$ depends only on this pair, it follows that
    \begin{equation*}
        \lim_{n \to \infty} a_n(y_n;F,G) = \E_0\!\sq{\finner{J}{F} \finner{J}{G}} = \frac{1}{d} \Tr(F^\top GL)
    \end{equation*}
    for any $y_n \in \cH_n$ and $\lim\limits_{n \to \infty} y_n = y$. This shows convergence of the second infinitesimal moments.
    
    Finally, every jump of $\bY^{(n)}$ is of the form $J/\sqrt{n}$ with $J \in \cJ$, so its Frobenius norm is at most $\sqrt{2d/n}$. Therefore, for every $\eps > 0$, the rate of jumps larger than $\eps$ is identically zero for all sufficiently large $n$. All the hypotheses of \cite{Dur}*{Theorem~8.7.1} are therefore satisfied, and hence
    \begin{equation*}
        (Y^{(n)}(t))_{t \in [0, T]} \Rightarrow (D(t))_{t \in [0, T]}.
    \end{equation*}
    
    It remains to determine the law of $D(t)$. For $D(0) = y_0$, the SDE \eqref{eq:sde} has the unique solution
    \begin{equation*}
        D(t) = y_0 P_t + \frac{1}{\sqrt{d}} \int_0^t (I - \Pi) \,\d W(s) L^{1/2} P_{t-s}, \qquad t \ge 0.
    \end{equation*}
    The stochastic integral is a centered Gaussian random variable on $\cH$. Hence, $D(t)$ is Gaussian with mean $y_0 P_t$. Define $Z(t) \colonequals D(t) - y_0 P_t$. Let $F, G \in \cH$. We have
    \begin{equation*}
        \finner{Z(t)}{F} = \frac{1}{\sqrt{d}} \int_0^t \finner{(I-\Pi)\,\d W(s)L^{1/2} P_{t-s}}{F} = \frac{1}{\sqrt{d}} \sum_{i=1}^d \sum_{j=1}^d \int_0^t (F P_{t-s}^\top L^{1/2})_{ij}\,\d W_{ij}(s).
    \end{equation*}
    Since the entries of $W$ are independent standard Brownian motions, It\^o's isometry gives
    \begin{align*}
        \Cov\brac{\finner{D(t)}{F}, \finner{D(t)}{G}} &= \E[\finner{Z(t)}{F} \finner{Z(t)}{G}]\\
        &= \frac{1}{d} \sum_{i,j,k,\ell \in [d]}
        \ind_{\{i=k,\ j=\ell\}}
        \int_0^t
        (F P_{t-s}^\top L^{1/2})_{ij}
        (G P_{t-s}^\top L^{1/2})_{k\ell}\,\d s\\
        &= \frac1d\int_0^t
        \finner{F P_{t-s}^\top L^{1/2}}{GP_{t-s}^\top L^{1/2}}\,\d s,\\
        &= \frac{1}{d} \int_0^t \Tr(L^{1/2} P_{t-s} F^\top G P_{t-s}^\top L^{1/2}) \d s\\
        &= \frac1d\int_0^t
        \Tr(F^\top G P_{t-s}^\top L P_{t-s}) \d s,\\
        &= \frac1d\int_0^t
        \Tr(F^\top G P_s^\top L P_s) \d s,
    \end{align*}
    Now we have
    \begin{equation*}
        \int_0^t P_s^\top L P_s\,\d s = - \int_0^t \frac{\d}{\d s}(P_s^\top P_s) \,\d s = I - P_t^\top P_t.
    \end{equation*}
    Thus, if $\Sigma_t \colon \cH \to \cH$ is defined by
    \begin{equation*}
        \Sigma_t F \colonequals \frac{1}{d} F (I - P_t^\top P_t),
    \end{equation*}
    then
    \begin{equation*}
        \Cov \bigl(\finner{D(t)}{F}, \finner{D(t)}{G}\bigr) = \finner{F}{\Sigma_t G}.
    \end{equation*}
    Consequently, we obtain $\Law(D(t)) = \cN_{\cH}(y_0 e^{t(U-I)}, \Sigma_t)$.
    
    For the final assertion, double stochasticity of $U$ and hence $P_t$ gives
    \begin{equation*}
        P_t^\top P_t = \Pi + (P_t-\Pi)^\top (P_t-\Pi).
    \end{equation*}
    Let $F \in \cH$. Then $F \Pi = 0$, and so the formula for $\Sigma_t$ yields
    \begin{equation*}
        (\Sigma_t-d^{-1}I_\cH) F = -\frac{1}{d} F(P_t-\Pi)^\top (P_t-\Pi).
    \end{equation*}
    Taking Frobenius norms and using submultiplicativity together with \eqref{eq:spectral-two-sided} yields the result.
\end{proof}

We now show that the stationary distribution of the centered and rescaled process $\bY^{(n)}$ converges to the stationary Gaussian law of the limiting diffusion.

\begin{lem}
\label{lem:stationary-limit}
    Let $X^{(n)}(\infty)$ have distribution $\pi_n$, and set
    \begin{equation*}
        Y^{(n)}(\infty)
        \colonequals
        \frac{X^{(n)}(\infty)-n\Pi}{\sqrt{n}}.
    \end{equation*}
    Then, as $n\to\infty$,
    \begin{equation*}
        Y^{(n)}(\infty)
        \Rightarrow
        \cN_\cH(0,d^{-1}I_\cH).
    \end{equation*}
\end{lem}

\begin{proof}
    By the Cram\'er--Wold theorem, it suffices to show that
    \begin{equation*}
        \finner{Y^{(n)}(\infty)}{A}
        \Rightarrow
        \cN(0,d^{-1}\fnorm{A}^2)
    \end{equation*}
    for every fixed $A=(a_{i,j})\in\cH$. The case $A=0$ is
    immediate, so assume that $A\neq0$.

    Label the $dn$ balls and the $dn$ urn slots by $[dn]$, with $n$ slots in each urn. Let $c(r)$ denote the color of ball $r$ and $u(s)$ the urn containing slot $s$. Assign ball $r$ to slot $p(r)$, where $p$ is a uniformly random permutation of $[dn]$. Each update in the chain $\bX^{(n)}$ permutes the contents of the selected slots and hence preserves the uniform assignment law. Its projection onto color counts is therefore stationary and, by uniqueness, has distribution $\pi_n$.

    Set $b_{r,s} \colonequals a_{c(r),u(s)}$ for $r, s \in [dn]$. Since $A \in \cH$, the array $(b_{r,s})$ has zero row and column sums, and $\finner{\Pi}{A} = 0$. Consequently, under the above realization of $X^{(n)}(\infty)$, we have
    \begin{equation*}
        \finner{Y^{(n)}(\infty)}{A}
        =\frac1{\sqrt{n}}\sum_{r=1}^{dn}b_{r,p(r)}.
    \end{equation*}
    Each entry $a_{i,j}$ occurs $n^2$ times in $(b_{r,s})$. Thus, the variance formula in \cite{CF}*{Theorem~1.1} gives
    \begin{align*}
        \sigma_n^2 \colonequals \Var\brac{\finner{Y^{(n)}(\infty)}{A}} = \frac{1}{n(dn-1)} \sum_{r,s=1}^{dn}b_{r,s}^2 = \frac{n}{dn-1} \fnorm{A}^2 \xrightarrow[n\to\infty]{} d^{-1}\fnorm{A}^2 > 0.
    \end{align*}

    We apply the normal approximation bound in the same theorem to the array $(b_{r,s}/(\sqrt{n}\,\sigma_n))$. Its permutation sum $\sum_{r=1}^{dn} (b_{r,p(r)}/(\sqrt{n}\,\sigma_n))$ has variance one, and
    \begin{equation*}
        \frac{1}{dn} \sum_{r,s \in [dn]}
        \abs{\frac{b_{r,s}}{\sqrt{n} \sigma_n}}^3 = \frac1{d \sqrt{n} \sigma_n^3} \sum_{i,j \in [d]} \abs{a_{i,j}}^3 \xrightarrow[n\to\infty]{} 0.
    \end{equation*}
    It follows that $\sigma_n^{-1}\finner{Y^{(n)}(\infty)}{A} \Rightarrow \cN(0,1)$. The limit of $\sigma_n^2$ now yields the result.
\end{proof}

\section{The cutoff profile}
\label{sec:profile}

To complete the proof of Theorem~\ref{thm:main}, we need to couple two copies of the chain started near the stationary mean and close to each other so that they coalesce quickly. For this, we use a slight variant of Proposition 3.10 of \cite{BBL}.

\begin{prop}
\label{prop:BBL}
    Set $\delta_1 = 1/(2d)$ and define
    \begin{equation*}
        B_n \colonequals \set{x \in \Omega_n}{\fnorm{x-n\Pi} \le \frac{n\delta_1}{2}}.
    \end{equation*}
    There exists $C_0 > 0$ such that, for every fixed $0 < \eps < \frac{\delta_1}{4\sqrt{2d}}$, if $n$ is sufficiently large, then
    \begin{equation*}
        \tv\brac{H_\eps^{(n)}(x,\cdot), H_\eps^{(n)}(y,\cdot)} \le \frac{C_0}{\sqrt{n\eps}}  \fnorm{x-y}
    \end{equation*}
    for all $x, y \in B_n$.
\end{prop}

\begin{proof}
    The proof is similar to that of Proposition 3.10 of \cite{BBL}. We present the full proof for completeness. Let $p_n$ be the metric projection onto the closed Frobenius ball of radius $n\delta_1$ centered at $n\Pi$. Consider the modified chain $\widehat{\bX}^{(n)}$ on $\Z^{d\times d}\cap(n\Pi+\cH)$ with transition rates
    \begin{equation*}
        \hq_n(x,x+J)
        \colonequals q_n(p_n(x),p_n(x)+J),
        \qquad J\in\cJ\setminus\{0\},
    \end{equation*}
    and write $\hH_t^{(n)}$ for its transition kernel. Since $p_n(x)$ has all row and column sums equal to $n$ and every entry at least $n/(2d)$, the extension of the rate formula~\eqref{eqn:trans} from Section~\ref{subsec:basic-prop} applies for all sufficiently large $n$. These rates agree with those of the original chain whenever $\fnorm{x-n\Pi} \le n \delta_1$, and their sum is at most $n$. Moreover, taking $\delta_0 = \delta_1$ in \eqref{rate-blue} gives
    \begin{equation}
    \label{eqn:blue}
        \hq_n(x,x+J)\ge nr_0
    \end{equation}
    for every state $x$ and every $J \in \cJ \setminus \{0\}$.

    For nonnegative real matrices $a,b$ with all row and column sums equal to $n$, \eqref{eqn:trans} gives
    \begin{equation*}
        \sum_{J \in \cJ \setminus \{0\}} \abs{q_n(a,a+J) - q_n(b,b+J)} \le \onenorm{a-b}.
    \end{equation*}
    Since $p_n$ is $1$-Lipschitz, it follows that
    \begin{equation}
    \label{red-rates}
        \sum_{J \in \cJ \setminus \{0\}} \abs{\hq_n(x, x+J)-\hq_n(y, y+J)} \le \onenorm{p_n(x) - p_n(y)} \le d\fnorm{x-y}
    \end{equation}
    for all states $x,y$ of the modified chain. In particular, if $x-y$ is a sum of at most $k$ signed jumps from $\cJ$, the right-hand side is at most $kr_1$, where $r_1\colonequals d\sqrt{2d}$. Define
    \begin{equation*}
        \Delta_n(t) \colonequals \sup_{x\in\Z^{d\times d}\cap(n\Pi+\cH),\, J\in\cJ\setminus\{0\}} \tv\brac{\hH_t^{(n)}(x,\cdot), \hH_t^{(n)}(x+J,\cdot)}.
    \end{equation*}
    Our main claim is that there exists a constant $C_0 > 0$ such that, for $0 < t < 1/(8r_1)$,
    \begin{equation}
    \label{recursion-Delta}
        \Delta_n(t) \le \frac{C_0}{\sqrt{r_0nt}} + \frac{1}{4} \Delta_n(t/4).
    \end{equation}
    Proving \eqref{recursion-Delta} will imply the proposition. Indeed, it follows by induction that for each $K \in \Z_{\ge 0}$,
    \begin{equation*}
        \Delta_n(t) \le 4^{-(K+1)} \Delta_n(t/4^{K+1}) + \frac{C_0}{\sqrt{r_0nt}}\sum_{k=0}^{K} 2^{-k}.
    \end{equation*}
    Since $\Delta_n(u) \le 1$ for all $u \ge 0$, letting $K \to \infty$ yields $\Delta_n(t) \le 2C_0/\sqrt{r_0nt}$.

    For each $z \in B_n$, couple the original and modified chains started at $z$ so that they agree until their common first exit from the ball of radius $n\delta_1$ centered at $n\Pi$. Lemma~\ref{stay-close-center} gives
    \begin{equation*}
        \tv\brac{H_\eps^{(n)}(z,\cdot), \hH_\eps^{(n)}(z,\cdot)} \le e^{-cn}.
    \end{equation*}
    Since $\eps < 1/(8r_1)$, the triangle inequality along a path in the graph defining $\dis_n$ gives
    \begin{equation*}
        \tv\brac{H_\eps^{(n)}(x,\cdot), H_\eps^{(n)}(y,\cdot)} \le \Delta_n(\eps) \dis_n(x,y) + 2e^{-cn}.
    \end{equation*}
    for all $x, y \in B_n$. The case $x = y$ is immediate. Otherwise, $\dis_n(x,y) \ge 1$, so the exponential term can be absorbed into a constant multiple of $\dis_n(x,y)/\sqrt{n\eps}$ for all sufficiently large $n$. Lemma~\ref{alg-dist} then gives the desired bound.

    It remains to prove \eqref{recursion-Delta}. Set $s = 3t/4$ and consider two copies of the modified chain started at $x$ and $x + J$, where $J \in \cJ \setminus \{0\}$. By \eqref{eqn:blue}, we may reserve jumps by $J$ at rate $n r_0$ in each copy and call them \emph{blue}. Couple the remaining jumps so that they agree as often as possible. Call the matched jumps \emph{green} and the unmatched jumps \emph{red}. Thus, when the two copies are at states $x_1$ and $x_2$, the total red rate is
    \begin{equation*}
        \sum_{J' \in \cJ \setminus \{0\}} \abs{\hq_n(x_1, x_1+J')-\hq_n(x_2, x_2+J')}.
    \end{equation*}

    We first generate the blue jumps on $[0,s]$. Let $N,N'$ each have distribution $\Pois(nr_0s)$, and couple $N$ and $N'+1$ maximally. Let $E_s$ denote the event $\{N = N'+1\}$. On $E_s$, choose $N'$ common independent uniform jump times in $[0,s]$, together with one additional independent uniform jump time for the first copy. There exists a constant $C_0$ such that
    \begin{equation*}
        \P\brac{E_s^\complement} = \tv\brac{\Law(N), \Law(N'+1)} \le \frac{C_0}{\sqrt{nr_0s}}.
    \end{equation*}
    On the event $E_s$, couple the green and red jumps as explained above. Let $R(u)$ be the total number of red jumps up to time $u \le s$ on $E_s$, and set $R(u) \colonequals 0$ on $E_s^\complement$. On $E_s$, the blue processes differ by at most one jump, so the positions of the two copies at time $u$ differ by a sum of at most $R(u) + 1$ signed jumps. Hence, by \eqref{red-rates}, their total red rate is at most $r_1(R(u) + 1)$. Therefore, we have
    \begin{equation*}
        \dfrac{\text{d}}{\text{d}u} \E[R(u)] \le r_1 \E[R(u)+1]\ \Rightarrow\  \E\sq{R(s)} \le e^{r_1s} - 1 \le 2r_1s < \frac{1}{4}.
    \end{equation*}
    On $E_s$, the additional blue jump cancels the initial displacement by time $s$, so the two positions then differ by a sum of at most $R(s)$ signed jumps. By the semigroup property and the triangle inequality, we have
    \begin{align*}
        \tv\brac{\hH_t^{(n)}(x,\cdot), \hH_t^{(n)}(x+J,\cdot)} &\le \P\brac{E_s^\complement} + \E\sq{\ind_{E_s} \tv\brac{\hH_{t-s}^{(n)}(\hX^{(n)}(s),\cdot), \hH_{t-s}^{(n)}(\hX'^{(n)}(s),\cdot)}}\\
        &\le \frac{C_0}{\sqrt{nr_0s}} + \E\sq{R(s)} \Delta_n(t-s).
    \end{align*}
    Taking the supremum over $x$ and $J$, and using $s = 3t/4$, proves \eqref{recursion-Delta}.
\end{proof}

Under condition~1 of Theorem~\ref{thm:main}, for
\begin{equation*}
    z = (z_\lambda)_{\lambda\in\LS} \in \cT_U \colonequals \overline{\set{(e^{iu\Im(\lambda)})_{\lambda\in\LS}}{u\in\R}} \subseteq (S^1)^{\abs{\LS}},
\end{equation*}
and $\phi \in \R$, define
\begin{equation*}
    m_{\phi,z} \colonequals S \diag \brac{0, \brac{z_\lambda e^{(-\gamma + i\Im(\lambda)) \phi}}_{\lambda\in\LS}, 0} S^{-1}.
\end{equation*}
The definition of $\cT_U$ ensures that these matrices are real, and their row and column sums are zero, so $m_{\phi,z}\in\cH$. The next lemma shows that the limiting diffusion, started near a suitable deterministic state, is close in total variation to a Gaussian law with covariance $d^{-1}I_\cH$.

\begin{lem}
\label{Lem:small-tv}
    Assume that condition~1 of Theorem~\ref{thm:main} holds. For every $\theta \in \R$, there exists $K_\theta > 0$ such that, for every $\eps \in (0,1]$, the following holds for all sufficiently large $C$, uniformly over $z \in \cT_U$. If $D(0)\in\cH$ is deterministic and
    \begin{equation*}
        \fnorm{D(0)-m_{-C,z}}\le\eps e^{\gamma C},
    \end{equation*}
    then
    \begin{equation*}
        \sup_{0 \le \delta \le \eps} \tv \brac{\Law(D(C+\theta-\delta)), \cN_\cH(m_{\theta,z}, d^{-1}I_\cH)} \le K_\theta\eps.
    \end{equation*}
\end{lem}

\begin{proof}
    Fix $\theta \in \R$ and $\eps \in (0, 1]$. Let $\delta \in [0, \eps]$ and set $t = C + \theta - \delta$. We take $C > \eps - \theta$ so that $t > 0$. Condition~1 and \eqref{eq:spectral-two-sided} give $\fnorm{P_t-\Pi} = O(e^{-\gamma t})$. By the definition of $m_{\phi,z}$, we have
    \begin{equation*}
        m_{-C,z}P_t = m_{\theta-\delta,z}\quad \text{ and } \quad \fnorm{m_{\theta-\delta,z}-m_{\theta,z}} = O_\theta(\delta),
    \end{equation*}
    uniformly over $z \in \cT_U$. Since $D(0)-m_{-C,z}\in\cH$, we have $(D(0)-m_{-C,z}) \Pi = 0$, and so
    \begin{align}
    \label{close-means}
        \fnorm{D(0)P_t-m_{\theta,z}}
        &\le \fnorm{(D(0)-m_{-C,z})(P_t - \Pi)} + \fnorm{m_{\theta-\delta,z}-m_{\theta,z}}\nonumber \\
        &= O_\theta(\eps e^{\gamma C-\gamma t} + \delta) = O_\theta(\eps).
    \end{align}
    By Proposition~\ref{prop:diffusion-convergence}, $D(t)$ is Gaussian with mean $D(0)P_t$ and covariance $\Sigma_t$, where
    \begin{equation}
    \label{close-covs}
        \opnorm{\Sigma_t-d^{-1}I_\cH} = O(t^{2r} e^{-2\gamma t}) = O_\theta(e^{-2\gamma C}).
    \end{equation}
    For all sufficiently large $C$, we have $d\opnorm{\Sigma_t-d^{-1}I_\cH}<1/2$. Applying Corollary 3 from \cite{BU} on the space $\cH$, with reference covariance $d^{-1}I_\cH$, and using \eqref{close-means} and \eqref{close-covs}, gives
    \begin{align*}
        \tv \brac{\Law(D(t)), \cN_\cH(m_{\theta,z}, d^{-1}I_\cH)} &= O \brac{\fnorm{D(0)P_t - m_{\theta,z}} + \opnorm{\Sigma_t-d^{-1} I_\cH}}\\
        &= O_\theta(\eps + e^{-2\gamma C}).
    \end{align*}
    Taking $C$ sufficiently large so that $e^{-2\gamma C} \le \eps$ proves the lemma.
\end{proof}

We will also use the following formula for the total variation distance between Gaussian laws with equal covariance, given by Theorem~1 of~\cite{BU}.
\begin{equation}
\label{eq:gaussian-profile}
    \tv \brac{\cN_\cH(m_{\theta,z}, d^{-1}I_\cH), \cN_\cH(0, d^{-1}I_\cH)} = 2\Phi \brac{\frac{\sqrt{d}}{2} \fnorm{m_{\theta,z}}} - 1.
\end{equation}

\begin{lem}
\label{lem:subseq-profile}
    Assume that condition~1 of Theorem~\ref{thm:main} holds and that, for a subsequence $(n_k)_{k \ge 1}$ of the positive integers, the limit
    \begin{equation}
    \label{eqn:lim}
        \lim_{k \to \infty}
        \frac{m(t_{n_k,-C})-n_k\Pi}{\sqrt{n_k}}
    \end{equation}
    exists for every sufficiently large fixed $C$. Then, for every $\lambda \in \LS$, the limit
    \begin{equation*}
        z_\lambda \colonequals \lim_{k \to \infty} e^{it_{n_k,0}\Im(\lambda)}
    \end{equation*}
    exists, and $z = (z_\lambda)_{\lambda \in \LS}$ belongs to $\cT_U$. Moreover, the limit in \eqref{eqn:lim} equals $m_{-C,z}$, and for every $\theta \in \R$, we have
    \begin{equation*}
        \lim_{k\to\infty}d_{n_k}(t_{n_k,\theta}) = 2\Phi\brac{\frac{\sqrt{d}}{2}\fnorm{m_{\theta,z}}} - 1.
    \end{equation*}
\end{lem}

\begin{proof}
    Under condition~1, the Jordan decomposition of $U$ and Proposition~\ref{prop:concentration} together give
    \begin{equation*}
        \frac{m(t_{n_k,-C})-n_k\Pi}{\sqrt{n_k}} = e^{\gamma C}K_U(t_{n_k,-C}) + o_{k\to\infty}(1)
    \end{equation*}
    for every fixed $C$, where the error term goes to zero entrywise. Conjugating by $S^{-1}$ on the left and $S$ on the right, convergence of the rescaled means for any one such $C$ implies that
    \begin{equation*}
        \brac{e^{i\Im(\lambda) t_{n_k,0}}}_{\lambda \in \LS} \to (z_\lambda)_{\lambda \in \LS}
    \end{equation*}
    for some $z \in \cT_U$. Substituting this into the same expansion shows that the rescaled means converge to $m_{-C,z}$ for every fixed $C$.

    We now prove the profile claim along $(n_k)_{k\ge1}$. Fix $\theta \in \R$ and choose
    \begin{equation*}
        0 < \eps < \min\left\{1, \frac{\delta_1}{4\sqrt{2d}}, \frac{1}{4C_0^2}\right\},
    \end{equation*}
    where $C_0$ is the constant in Proposition~\ref{prop:BBL}. Choose $C$ sufficiently large for Lemma~\ref{Lem:small-tv} to apply and set
    \begin{equation*}
        \cR \colonequals \set{y\in\cH}{\fnorm{y-m_{-C,z}} \le \eps e^{\gamma C}}, \qquad \cR_k \colonequals \cR \cap \cH_{n_k}.
    \end{equation*}
    For fixed $C$ and all sufficiently large $k$, the rescaled mean at time $t_{n_k,-C}$ is within $\eps e^{\gamma C}/2$ of $m_{-C,z}$. Therefore, by Proposition~\ref{prop:concentration}, we have
    \begin{equation}
    \label{eq:profile-concentration}
        \P_{n_kI}\!\brac{Y^{(n_k)}(t_{n_k,-C}) \notin \cR_k} \le \eps.
    \end{equation}
    In what follows, $C$ is fixed before letting $k \to \infty$.

    \textbf{Proof of the upper bound.}
    We first establish the upper bound
    \begin{equation}
    \label{upper-bd}
        \limsup_{k\to\infty} \sup_{y \in \cR_k} \tv \brac{H_{C+\theta}^{(n_k)}(n_k\Pi+\sqrt{n_k}y,\cdot), \pi_{n_k}} \le 2\Phi \brac{\frac{\sqrt{d}}{2}\fnorm{m_{\theta,z}}} - 1 + (K_\theta+4)\eps.
    \end{equation}
    Suppose that \eqref{upper-bd} does not hold. After passing to a further subsequence if necessary, we may choose $y_{n_k} \in \cR_k$ for which the corresponding total variation distances exceed the right-hand side of \eqref{upper-bd} by a fixed positive amount. Since $\cR$ is compact, we may pass to a further subsequence such that $y_{n_k} \to y_0 \in \cR$.

    Set $T = C+\theta-\eps$. By Proposition~\ref{prop:diffusion-convergence}, for the chain started at $Y^{(n_k)}(0) = y_{n_k}$, we have
    \begin{equation*}
        Y^{(n_k)}(T) \Rightarrow D(T),
        \qquad D(0)=y_0,
    \end{equation*}
    where $D$ solves \eqref{eq:sde}. Moreover, Lemma~\ref{lem:stationary-limit} gives
    \begin{equation*}
        Y^{(n_k)}(\infty)\Rightarrow D(\infty),
        \qquad D(\infty)\sim\cN_\cH(0,d^{-1}I_\cH).
    \end{equation*}
    By Skorokhod's representation theorem, there exist couplings such that, for all sufficiently large $k$, we have
    \begin{align*}
        \P \brac{\fnorm{Y^{(n_k)}(T)-D(T)} > \eps^2} \le \eps, \quad \P \brac{\fnorm{Y^{(n_k)}(\infty)-D(\infty)} > \eps^2} \le \eps.
    \end{align*}
    By Lemma~\ref{Lem:small-tv}, with $\delta = \eps$, and \eqref{eq:gaussian-profile}, there exists a coupling of $D(T)$ and $D(\infty)$ such that
    \begin{equation*}
        \P\brac{D(T) = D(\infty)} \ge 2 - 2\Phi \brac{\frac{\sqrt{d}}{2}\fnorm{m_{\theta,z}}} - K_\theta \eps.
    \end{equation*}
    Choose $K\ge2$ such that $\P\brac{\fnorm{D(\infty)} > K/2} \le \eps$. For each sufficiently large $k$, combining these couplings gives a coupling of $Y^{(n_k)}(T)$ and $Y^{(n_k)}(\infty)$ for which the event
    \begin{equation*}
        E \colonequals \left\{\fnorm{Y^{(n_k)}(T)-Y^{(n_k)}(\infty)} \le 2\eps^2\right\} \cap \left\{\fnorm{Y^{(n_k)}(\infty)}\le K\right\}
    \end{equation*}
    satisfies
    \begin{equation*}
        \P\brac{E} \ge 2 - 2\Phi \brac{\frac{\sqrt{d}}{2}\fnorm{m_{\theta,z}}} - (K_\theta+3) \eps.
    \end{equation*}
    On $E$, both rescaled states have Frobenius norm at most $K+2\eps^2$, so their corresponding states belong to $B_{n_k}$ for all sufficiently large $k$. By Proposition~\ref{prop:BBL}, the total variation distance between their transition laws after an additional time $\eps$ is at most
    \begin{equation*}
        \frac{C_0}{\sqrt{n_k\eps}} \cdot 2\sqrt{n_k} \eps^2 = 2C_0 \eps^{3/2} \le \eps.
    \end{equation*}
    Therefore, the original chains started at $n_k \Pi + \sqrt{n_k} y_{n_k}$ and from $\pi_{n_k}$ can be coupled at time $C + \theta$ with failure probability at most
    \begin{equation*}
        \P\brac{E^\complement} + \eps \le 2\Phi\brac{\frac{\sqrt{d}}{2}\fnorm{m_{\theta,z}}} - 1 + (K_\theta+4) \eps.
    \end{equation*}
    This contradicts the choice of $y_{n_k}$ and proves \eqref{upper-bd}.

    Now, by Lemma~\ref{lem:worst-start} and the Markov property, we have
    \begin{equation*}
        d_{n_k}(t_{n_k,\theta}) \le \P_{n_kI}\!\brac{Y^{(n_k)}(t_{n_k,-C}) \notin \cR_k} + \sup_{y\in\cR_k} \tv \brac{H_{C+\theta}^{(n_k)}(n_k\Pi+\sqrt{n_k}y,\cdot), \pi_{n_k}}.
    \end{equation*}
    Combining \eqref{eq:profile-concentration} and \eqref{upper-bd}, and then letting $\eps \to 0^+$, gives
    \begin{equation*}
        \limsup_{k\to\infty}d_{n_k}(t_{n_k,\theta})
        \le2\Phi\brac{\frac{\sqrt{d}}{2}\fnorm{m_{\theta,z}}}-1.
    \end{equation*}

    \textbf{Proof of the lower bound.}
    The total variation distance in \eqref{eq:gaussian-profile} is realized by the half-space
    \begin{equation*}
        W \colonequals \set{x\in\cH}{\finner{x}{m_{\theta,z}} \ge \tfrac{1}{2}\fnorm{m_{\theta,z}}^2}.
    \end{equation*}
    Set $W_n = (n\Pi + \sqrt{n} W) \cap\Omega_n$. For sufficiently large $C$, since the law of $D(C + \theta)$ is absolutely continuous, $W$ is a continuity set. By compactness of $\cR$, every sequence $y_{n_k} \in \cR_k$ has a convergent subsequence with limit $y_0 \in \cR$. Along this subsequence, Proposition~\ref{prop:diffusion-convergence} gives convergence of the measures of $W$, while Lemma~\ref{Lem:small-tv}, with $\delta = 0$, bounds the limiting probability below by
    \begin{equation*}
        \cN_\cH(m_{\theta,z},d^{-1}I_\cH)(W)
            -K_\theta\eps.
    \end{equation*}
    It follows that
    \begin{equation*}
        \liminf_{k\to\infty}\inf_{y\in\cR_k}
                H_{C+\theta}^{(n_k)}
                    \left(n_k\Pi+\sqrt{n_k}y,W_{n_k}\right) \ge
                \cN_\cH(m_{\theta,z},d^{-1}I_\cH)(W)
                    -K_\theta\eps.
    \end{equation*}
    By the Markov property and \eqref{eq:profile-concentration}, for all sufficiently large $k$, we have
    \begin{equation*}
        H_{t_{n_k,\theta}}^{(n_k)}\left(n_kI,W_{n_k}\right) \ge (1-\eps) \inf_{y\in\cR_k} H_{C+\theta}^{(n_k)}(n_k\Pi+\sqrt{n_k}y,W_{n_k}).
    \end{equation*}
    Moreover, Lemma~\ref{lem:stationary-limit} gives
    \begin{equation*}
        \pi_{n_k}(W_{n_k}) \to \cN_\cH(0,d^{-1}I_\cH)(W).
    \end{equation*}
    Therefore, using \eqref{eq:gaussian-profile}, we obtain
    \begin{align*}
        \liminf_{k\to\infty} d_{n_k}(t_{n_k,\theta}) &\ge \liminf_{k\to\infty} \brac{H_{t_{n_k,\theta}}^{(n_k)}(n_kI,W_{n_k}) - \pi_{n_k}(W_{n_k})}\\
        &\ge (1-\eps) \brac{ \cN_\cH(m_{\theta,z},d^{-1}I_\cH)(W) - K_\theta\eps} -\cN_\cH(0,d^{-1}I_\cH)(W)\\
        &\ge 2\Phi\brac{\frac{\sqrt{d}}{2} \fnorm{m_{\theta,z}}} - 1 -(K_\theta+1)\eps.
    \end{align*}
    Letting $\eps \to 0^+$ proves the lower bound and completes the proof.
\end{proof}

\begin{proof}[Proof of Theorem~\ref{thm:main}]
    Fix $\theta \in \R$. By condition~1 and \eqref{eq:spectral-two-sided}, the sequence $\{\sqrt{n}(P_{t_{n,0}}-\Pi)\}$ is bounded. Hence, from any subsequence of the positive integers, we can extract a subsequence $(n_k)_{k\ge1}$ along which these matrices converge. For every fixed $C>0$, we have
    \begin{equation*}
        \frac{m(t_{n_k,-C})-n_k\Pi}{\sqrt{n_k}} = \sqrt{n_k}\,(P_{t_{n_k,0}}-\Pi) e^{-C(U-I)}
    \end{equation*}
    for all sufficiently large $k$. Thus, along the same subsequence, the rescaled means converge for every fixed $C > 0$. Lemma~\ref{lem:subseq-profile} then yields a phase vector $z$ which by continuity of the Frobenius norm and condition~2 satisfies
    \begin{equation*}
        \fnorm{m_{\theta,z}} = e^{-\gamma \theta} \lim_{k\to\infty} \fnorm{K_U(t_{n_k,\theta})} = e^{-\gamma\theta} \kappa_U.
    \end{equation*}
    Consequently, the final assertion in Lemma~\ref{lem:subseq-profile} gives
    \begin{equation*}
        \lim_{k \to \infty} d_{n_k}(t_{n_k,\theta}) = 2\Phi\brac{\frac{\kappa_U \sqrt{d}}{2} e^{-\gamma \theta}} - 1.
    \end{equation*}
    Since every subsequence has a further subsequence along which the total variation distances converge to this same value, the whole sequence converges to the claimed limit.
\end{proof}

\bibliographystyle{amsplain}
\bibliography{references}

\end{document}